\documentclass[11pt,english]{amsart}

\def\margin_comment#1{\marginpar{\sffamily{\tiny #1\par}\normalfont}}
\usepackage{faktor}

\usepackage[all]{xy}
\usepackage{amssymb,euscript}
\usepackage{amsfonts}
\usepackage{geometry}
\usepackage{fancyhdr}

\usepackage[table]{xcolor}
\definecolor{fucsia}{RGB}{196,0,98}
\usepackage{hhline}
\usepackage{colortbl}
\newcommand{\rvline}{\hspace*{-\arraycolsep}\vline\hspace*{-\arraycolsep}}
\newgeometry{left=3cm,right=3cm,top=3.5cm,bottom=3cm}

\usepackage{tikz}
\usepackage{pgfplots}
\pgfplotsset{compat=1.13}
\usepackage{braids}
\usepackage{graphicx}
\usepackage{setspace}
\usepackage{color}
\usepackage{amsmath,amssymb,amsthm}
\usepackage{epsfig}

\usepackage{amsmath}
\usepackage{tikz-cd}
\usetikzlibrary{matrix, calc, arrows}
\usepackage{tikz}
\usetikzlibrary{arrows,chains,matrix,positioning,scopes}
\usepackage{lipsum}
\usepackage{mathtools}
\usepackage{mathrsfs}
\usepackage{graphicx}
\usepackage{graphics}
\usepackage{setspace}
\usepackage{color}
\usepackage{mwe}

\usepackage{amsmath,amssymb,amsthm}
\usepackage{epsfig}
\usepackage{babel}
\usepackage{tikz}
\usepackage{tikz-cd}
\usepackage{enumerate}
\usepackage{float}
\usepackage[utf8]{inputenc}
\usepackage[T1]{fontenc}
\usepackage[document]{ragged2e}
\usepackage{verbatim}
\usepackage{amssymb,euscript}
\usepackage{amsfonts}
\usepackage{geometry}
\usepackage{fancyhdr}
\usepackage{graphicx}
\usepackage{setspace}
\usepackage{color}
\usepackage{amsmath,amssymb,amsthm}
\usepackage{epsfig}
\usepackage{babel}
\usepackage{tikz}
\usepackage{amsmath}
\usepackage{tikz-cd}
\usetikzlibrary{matrix, calc, arrows}
\usepackage{tikz}
\usetikzlibrary{arrows,chains,matrix,positioning,scopes}
\usepackage{lipsum}
\usepackage{mathtools}
\usepackage{mathrsfs}
\usepackage{graphicx}
\usepackage{graphics}
\usepackage{setspace}
\usepackage{color}

\usetikzlibrary{trees}
\usetikzlibrary{arrows,shapes,automata,backgrounds,petri,positioning,scopes,matrix, calc, plothandlers,plotmarks,patterns}
\usepackage{enumerate}
\usepackage{float}
\usepackage[utf8]{inputenc}
\usepackage[T1]{fontenc}
\usepackage[document]{ragged2e}
\usepackage{verbatim}

\usepackage{amssymb,euscript}
\usepackage{amsfonts}
\usepackage{geometry}
\usepackage{fancyhdr}
\usepackage{graphicx}
\usepackage{setspace}
\usepackage{color}
\usepackage{amsmath,amssymb,amsthm}
\usepackage{epsfig}
\usepackage{babel}
\usepackage{tikz}
 \usepackage{amsmath}
\usepackage{tikz-cd}
\usetikzlibrary{matrix, calc, arrows}
\usepackage{tikz}
\usetikzlibrary{arrows,chains,matrix,positioning,scopes}
\usepackage{lipsum}
\usepackage{mathtools}
\usepackage{mathrsfs}
\usepackage{graphicx}
\usepackage{graphics}
\usepackage{setspace}
	\usepackage{color}

\usetikzlibrary{trees}
\usetikzlibrary{arrows,shapes,automata,backgrounds,petri,positioning,scopes,matrix, calc, plothandlers,plotmarks,patterns}
\usepackage{enumerate}
\usepackage{float}
\usepackage[utf8]{inputenc}
\usepackage[T1]{fontenc}
\usepackage[document]{ragged2e}
 \usepackage{verbatim}
\newtheorem{thm}{Theorem}[section]
\numberwithin{equation}{section} %% Comment out for sequentially-numbered
\numberwithin{figure}{section} %% Comment out for sequentially-numbered
\theoremstyle{plain}
\newtheorem*{thm*}{Theorem}
\theoremstyle{definition}
\theoremstyle{plain}
\newtheorem{thm_A}{Theorem}
\newtheorem*{defn*}{Definition}

\theoremstyle{plain}

\theoremstyle{plain} %%Delete [thm] to re-start numbering

\theoremstyle{plain}

\newtheorem{prop}[thm]{Proposition} %%Delete [thm] to re-start numbering
\theoremstyle{remark}
\newtheorem{ex}[thm]{Example}
\theoremstyle{remark}
\newtheorem{rem}[thm]{Remark}
\theoremstyle{plain}

\theoremstyle{plain}

\theoremstyle{plain}
\newtheorem{lem}[thm]{Lemma} %%Delete [thm] to re-start numbering
\theoremstyle{definition}
\newtheorem{defn}[thm]{Definition}

\newtheorem*{acknowledgment*}{Addentum}

\theoremstyle{plain}
\newtheorem*{ex*}{Example}
\theoremstyle{plain}
\begin{document}
%\doublespacing
\title[Garsideness properties of structure groups]{On some Garsideness properties of structure groups of set-theoretic solutions of the Yang-Baxter equation}
\author{Fabienne Chouraqui}
%\date{\today}
\begin{abstract}
There exists a multiplicative homomorphism from the braid group $B_{k+1}$  on $k +1$ strands to the Temperley–Lieb algebra $TL_k$. Moreover,  the homomorphic images in $TL_k$  of the simple elements   form a basis  for the vector space underlying $TL_k$.
In analogy with  the case of $B_k$, there exists a multiplicative homomorphism from  the structure group $G(X,r)$ of  a   non-degenerate, involutive set-theoretic solution $(X,r)$,  with $\mid X \mid =n$, to an algebra, which  extends to a homomorphism of algebras. We construct a finite basis of the underlying vector space of  the  image of  $G(X,r)$ using the Garsideness properties of $G(X,r)$.

\end{abstract}
\maketitle
%\tableofcontents
%%%%%%%%%%%%%%%%%%%%%%%%%%%%%%%%%%%%%%%%%%%%%%%
%%%%%%%%%%%%%%%%%%%%%%%%%%%%%%%%%%%%%%%%%%%%%%%
%%%%%%%%%%%%%%%%%%%%%%%%%%%%%%%%%%%%%%%%%%%%%%%%%%%%%%%%%%%%%%
\section*{Introduction}
The quantum Yang-Baxter equation is an equation in mathematical physics and it lies in the  foundation of  the theory of quantum groups. One of the fundamental problems is to find all the solutions of this equation.   Drinfeld suggested the study of a particular class of solutions,  derived from the so-called set-theoretic solutions  \cite{drinf}.  A  set-theoretic solution of the Yang-Baxter equation is a pair $(X,r)$, where $X$ is a set and 
\[r: X \times X \rightarrow X \times X\,,\;\;\; r(x,y)=(\sigma_x(y),\gamma_y(x))\]
is a bijective map satisfying $r^{12}r^{23}r^{12}=r^{23}r^{12}r^{23}$, where $r^{12}=r \times Id_X$,  $r^{23}=Id_X\times r$. 
            
 Non-degenerate and involutive set-theoretic solutions of the quantum Yang-Baxter equation are  intensively investigated and they give rise to several algebraic structures associated to them. One of these  is the structure group of a solution, $G(X,r)$, which is defined by $G(X,r)=\operatorname{Gp}\langle X\mid x_ix_j=x_kx_l\,\,; r(x_i,x_j)=(x_k,x_l)\rangle$ in \cite{etingof}. 
             
  \setlength\parindent{10pt} There is a one-to-one correspondence between the structure groups of non-degenerate,  involutive  set-theoretic solutions and a particular class of groups, the so-called Garside groups, with a particular presentation \cite{chou_art}.  Garside groups  have been defined as a generalization in some sense of the braid groups, and the finite-type Artin groups \cite{DePa}.  A  monoid $M$ is \emph{Garside}  if it is cancellative with  $1$  the unique invertible element, it  admits  a structure of lattice with respect to left-divisibility and right-divisibility, and there exists a balanced  element $\Delta$  (i.e. its set of left and right divisors coincide, and  is denoted by $\operatorname{Div}(\Delta)$),    with  $\operatorname{Div}(\Delta)$ a finite generating set of $M$.  A Garside monoid admits a left and right group of fractions which  coincide and is called a \emph{Garside group}.\\

  \setlength\parindent{10pt} There exists an exact  sequence $1 \rightarrow P_k \rightarrow B_k  \rightarrow S_k \rightarrow 1$, where $B_k$ is the braid group and $P_k$ the pure braid group and  more generally there exists  an exact  sequence $1 \rightarrow N \rightarrow A \rightarrow W\rightarrow 1$, where  $A$ is an Artin group of finite type and $W$ the corresponding finite Coxeter group. Furthermore,  there is a bijection between the elements in $W$ and the set $\operatorname{Div}(\Delta)$. In the case of $B_k$, the  elements in $\operatorname{Div}(\Delta)$   are called \emph{simple elements} and they can be seen as lifts of noncrossing partitions,  viewed as elements of the symmetric group $S_k$  to $B_k$. \\

  \setlength\parindent{10pt} In  \cite{chou_godel2}, we attempt to answer in the context of the structure groups, the question  posed in \cite{bessis2},  whether any Garside group admits a Coxeter-like quotient group, that is a finite quotient group that plays the role  the symmetric group $S_k$  plays for   $B_k$ or the role Coxeter groups play for the finite-type Artin groups.  We give  a partial  answer to this question. Indeed,  we show that  if  $(X, r)$ is  a  solution, with $\mid X \mid =n$, that satisfies  a certain condition $(C)$, there is a short exact sequence $1 \rightarrow N \rightarrow G(X,r)  \rightarrow W  \rightarrow 1$, where $N$ is a normal free abelian subgroup of rank $n$ and $W$ is a finite group of order $2^n$.  Moreover, $W$  is  a Coxeter-like group and there is a bijection between the elements in $W$ and the set $\operatorname{Div}(\Delta)$.  In \cite{deh_coxeterlike}, P. Dehornoy  proves that  the condition $(C)$ may be dropped and that, in   general,  there exists  a Coxeter-like  group $W$ defined by $G(X,S)/N$  of order $m^n$, where $m\geq 1$  is an integer called \emph{the class of the solution}.  There is a bijection between the elements in $W$ and the set $\operatorname{Div}(\Delta^{m-1})$,   $N=\langle \theta_1,...,\theta_n\rangle$ is  a free abelian  subgroup of  $G(X,r)$ of rank $n$, and $\theta_1,...,\theta_n$ are called \emph{the  frozen elements of length $m$}    \cite{deh_coxeterlike}.

  \setlength\parindent{10pt} There is a multiplicative homomorphism from the braid group $B_{n+1}$  on $n +1$ strands to the Temperley–Lieb algebra $TL_n$.   The Temperley–Lieb algebra $TL_n$ (of type $A_n$)   is an associative, unital $\mathbb{Z}[v,v^{-1}]$ 
algebra of dimension equal to the $(n + 1)$th Catalan number.  It was originally defined by Temperley and Lieb (see \cite{temperley}).  Alternatively, it can be viewed as a quotient algebra of the Iwahori–Hecke algebra $H$ of type $A_n$.  Several bases for the vector space underlying the Temperley-Lieb algebra  are known.  There is a particular one based on the Garsideness of the dual braid monoid   \cite{zinno}. Indeed, it is proved there that the homomorphic images in $TL_n$ of the simple elements   form a $\mathbb{Z}[v,v^{-1}]$-linear basis  for the vector space underlying the Temperley-Lieb algebra.  In \cite{gobet},  some results from \cite{zinno} are presented in a simpler way,  and   an additional   basis is found. 

  \setlength\parindent{10pt} In this paper, we study the question whether a result of the kind obtained in \cite{zinno,gobet} can be obtained for other Garside groups and in particular for structure groups of solutions. In analogy with  the case of $B_k$, there exists a multiplicative homomorphism from  the structure group $G(X,r)$ of  a   non-degenerate, involutive set-theoretic solution $(X,r)$,  with $\mid X \mid =n$, to a matrix algebra.  Indeed, $G(X,r)$ is a Bieberbach group of rank  $n$, that  is a torsion-free crystallographic group   \cite{gateva_van}, \cite{jespers_book}[p.218]. So, from  the action of $G(X,r)$,  there is an embedding 	$\psi: \; G(X,r)  \; \rightarrow    \; GL_{n+1}(\mathbb{R})$,  with a  permutation matrix  in the upper $n\times n$ block and a translation part in the last column, and by linear extension $\psi$ induces  a  unital homomorphism of  $\mathbb{R}$-algebras $\hat{\psi}:  \mathbb{R}G  \rightarrow    M_{n+1}(\mathbb{R}) $, where $\mathbb{R}G$ denotes the group algebra of $G(X,r)$.
 
 \begin{figure}[h]
 	\centering
 	\begin{equation}\label{eqn-psi-gp-alg}
 	\hat{\psi}: \; \mathbb{R}G  \; \rightarrow    \; M_{n+1}(\mathbb{R}) 
 	\end{equation}
 	\[  \sum\limits_{g \in G}k_g\,g  \;    \mapsto\; 	 \sum\limits_{g \in G}k_g\,\psi(g) 	\]
 \end{figure}
 The basis of   $\mathbb{R}G$ is the set of all group elements $g \in G(X,r)$ and every element in  $\mathbb{R}G$  can be written as    $\sum\limits_{g \in G}k_g\,g$, with $k_g\neq 0$ for only  finitely many elements $g \in G(X,r)$.
 The subalgebra  of  $M_{n+1}(\mathbb{R}) $, $\operatorname{Im}(\hat{\psi})=\{\sum\limits_{g \in G}k_g\,\psi(g) 	 \mid k_g \in \mathbb{R}\}$  is spanned by  the infinite set $\{\psi(g) 	\mid g \in G(X,r)\}$. However,  we show that it is finite-dimensional.  We prove the following:
 % The authors prove that, $G(X,r)$,  the structure group  of a non-degenerate, involutive set-theoretic solution  $(X,S)$ embeds into the semidirect product $\mathbb{Z}^X\rtimes \operatorname{Sym}_X$, where $\operatorname{Sym}_X$ denotes the symmetric group of $X$ and  $\mathbb{Z}^X$  is the free abelian group generated by $X$.   Moreover they prove that if $X$ is finite, then $G(X,r)$ is a solvable group \cite{etingof}. \\
 \begin{thm_A}\label{theo}
	Let $(X,r)$ be  a  non-degenerate involutive set-theoretic solution of the  quantum Yang-Baxter equation with  structure group $G(X,r)$.  Let $m \geq 2$ be the class of the solution $(X,r)$.  
 	Let $\mathcal{S}=\operatorname{Div}(\Delta^{m-1})$ and $\theta_1,...,\theta_n$ be  the  $n$  frozen elements of length $m$.  Let $\psi$ and $\hat{\psi}$ be defined as above. Then the underlying vector space of 
$\operatorname{Im}(\hat{\psi})$ is  finite-dimensional  and it has  a basis  a subset of  the set $\{\psi(s)\mid s \in \mathcal{S}\}  \cup \{\psi(\theta_i)\mid 1 \leq i \leq n\}$.
 \end{thm_A}

 The paper is organized as follows. In Section $1$, we give some preliminaries on set-theoretic solutions of the Yang-Baxter equation, on braces and on  Garside groups. In Section $2$,  we describe the specific Garside structure of structure groups of  non-degenerate  involutive set-theoretic solutions.  In Section $3$,    Theorem  \ref{theo}  is reformulated more precisely and we prove it.

\section{Preliminaries}
%%%%%%%%%%%%%%%%%%%%%%%%%%%%%%%%%%%%%%%%%%%%%%%
\subsection{Definition and properties of set-theoretic solutions of the QYBE}
%%%%%%%%%%%%%%%%%%%%%%%%%%%%%%%%%%%%%%%%%%%%%%%
\label{subsec_qybe_Backgd}We  refer to \cite{etingof}, \cite{gateva_van,gateva_new},  \cite{jespers_book}.\\
Let $X$ be a non-empty set. Let $r: X \times X \rightarrow X \times X$  be a map and write $r(x,y)=(\sigma_{x}(y),\gamma_{y}(x))$,  where $\sigma_x, \gamma_x:X\to X$ are functions  for all  $x,y \in X$.   The pair $(X,r)$ is  \emph{braided} if $r^{12}r^{23}r^{12}=r^{23}r^{12}r^{23}$, where the map $r^{ii+1}$ means $r$ acting on the $i$-th and $(i+1)$-th components of $X^3$.  In this case, we  call  $(X,r)$  \emph{a set-theoretic solution of the QYBE}, and whenever $X$ is finite, we  call  $(X,r)$  \emph{a finite set-theoretic solution of the QYBE}.  The pair $(X,r)$ is \emph{non-degenerate} if for every  $x\in X$,  $\sigma_{x}$ and $\gamma_{x}$  are bijective and it   is  \emph{involutive} if $r\circ r = Id_{X^2}$. If $(X,r)$ is a non-degenerate involutive set-theoretic solution, then $r(x,y)$ can be described as  $r(x,y)=(\sigma_{x}(y),\gamma_{y}(x))=(\sigma_{x}(y),\,\sigma^{-1}_{\sigma_{x}(y)}(x))$.  A set-theoretic solution  $(X,r)$ is \emph{square-free},  if  for every $x \in X$, $r(x,x)=(x,x)$. A set-theoretic  solution $(X,r)$ is \emph{trivial} if $\sigma_{x}=\gamma_{x}=Id_X$, for every  $x \in X$. 
 \begin{defn}
 	The \emph{retract} relation $\sim$ on the set $X$ is defined by $x \sim y$ if $\sigma_x=\sigma_y$. There is a natural induced solution $Ret(X,r)=(X/\sim,r)$, called the \emph{the retraction of $(X,r)$}, defined by $r'([x],[y])=([\sigma_{x}(y)],[\gamma_y(x)])$.
 	A non-degenerate involutive set-theoretic solution  $(X,r)$   is called \emph{a multipermutation solution of level $m$} if $m$ is the smallest natural number such that the solution $\mid Ret^m(X,r)\mid=1$, where 
 	$Ret^k(X,r)=Ret(Ret^{k-1}(X,r))$, for $k>1$. If such an $m$ exists, $(X,r)$   is also called \emph{retractable}, otherwise it is called \emph{irretractable}.
 \end{defn}
\begin{defn}
	Let  $(X,r)$   be a set-theoretic solution of the QYBE. The \emph{structure group} of $(X,r)$ is  defined by $G(X,r)=\operatorname{Gp} \langle X\mid\ xy =\sigma_x(y)\gamma_y(x)\ ;\ x,y\in X \rangle$.  
\end{defn}
The  structure group of the trivial solution is $\mathbb{Z}^{X }$.  Two  set-theoretic solutions $(X,r)$ and $(X',r')$ are \emph{isomorphic} if there is a bijection $\alpha:X \rightarrow X'$ such that $(\alpha \times \alpha) \circ r=r'\circ (\alpha \times \alpha)$  \cite{etingof}. If  $(X,r)$ and $(X',r')$ are isomorphic, then $G(X,r) \simeq G(X',r')$, with  $G(X,r)$ and $G(X',r')$ their respective structure groups. 
An important characterisation of  non-degenerate involutive set-theoretic solutions of the QYBE is  presented in the following proposition.
\begin{thm}\cite[p.176-180]{etingof}\label{prop-etingof}
Let  $(X,r)$ be a non-degenerate involutive set-theoretic solution of the QYBE, defined by  $r(x,y)=(\sigma_{x}(y),\gamma_{y}(x))$,   $x,y \in X$,  with structure group $G(X,r)$. Let $\mathbb{Z}^{X}$ denote the free abelian group with basis $\{t_x \mid x \in X\}$, and  $\operatorname{Sym}_X$ denote the symmetric group of $X$. Then 
\begin{enumerate}[(i)]
		\item The map $\phi: G(X,r) \rightarrow \operatorname{Sym}_X$, defined by $x \mapsto \sigma_{x}$,  is a homomorphism of groups.
		\item  The group $\operatorname{Sym}_X$  acts on $\mathbb{Z}^{X}$.
		\item The group $G(X,r) $  acts on $\mathbb{Z}^{X}$:  if  $g \in G$, then $g \bullet t_x=t_{\alpha(x)}$, with $\alpha=\phi(g)$. 
	\item The map 	$\pi:G(X,r)\rightarrow \mathbb{Z}^{X}$ is a bijective $1$-cocycle, where  $\pi(x)=t_x$,  for $x \in X$, and  $\pi(gh) =\pi(g)+g \bullet \pi(h)$,  for $g,h \in G(X,r)$.

 \item There  is a monomorphism of groups $\tilde{\Phi} : G(X,r) \rightarrow \mathbb{Z}^{X} \rtimes \operatorname{Sym}_X$:  $\tilde{\Phi} (x)=(t_x,\sigma_x) $, $\tilde{\Phi} (g)=(\pi(g), \phi(g))$.
 
 \item  The group $G(X,r) $ is isomorphic to  a subgroup of  $\mathbb{Z}^{X} \rtimes \operatorname{Sym}_X$ of the form $H=\{(a,\varphi(a))\mid a \in \mathbb{Z}^{X} \}$, where $\varphi: \mathbb{Z}^{X} \rightarrow \operatorname{Sym}_X$, is defined by $\varphi(a)=\phi(g)$,  whenever  $\pi(g)=a$.
 \end{enumerate}
\end{thm}
The subgroup of $\operatorname{Sym}_X$  generated by $\{\sigma_x\mid x\in X\}$ is denoted by $\mathcal{G}(X,r)$ and is called \emph{a IYB group} \cite{cedo}.

\begin{lem}\cite{etingof}\label{lem_etingof_inverseT}
Let $D: X \rightarrow X$ be the map defined by $D(x) =\sigma^{-1}_x(x)$. Then the  map $D$ is invertible and  $D^{-1}(y)= \gamma^{-1}_y(y)$,  with $x,y \in X$ such that   $r(x,y)=(x,y)$.
Inductively, $D^{m}(x)=\sigma^{-1}_{D^{m-1}(x)}D^{m-1}(x)=\sigma^{-1}_{D^{m-1}(x)}\sigma^{-1}_{D^{m-2}(x)}..\sigma^{-1}_{D(x)}\sigma^{-1}_{x}(x)$.
\end{lem}

\begin{ex} \label{exemple:exesolu_et_gars} Let $X = \{x_1,x_2,x_3,x_4\}$, and  $r: X\times X\to X\times X$ be defined by $r(x_i,x_j)=(x_{\sigma_{i}(j)},x_{\gamma_{j}(i)})$ where $\sigma_i$ and $\gamma_j$ are permutations on $\{1,2,3,4\}$: $\sigma_{1}=(3,4)$, $\sigma_2=(1,4,2,3)$, $\sigma_{3}=(2,1)$, $\sigma_{4}=(3,2,4,1)$; $\gamma_{1}=(2,3)$, $\gamma_2=(2,1,3,4)$, $\gamma_{3}=(4,1)$, $\gamma_{4}=(4,3,1,2)$.
	Then, $G=\operatorname{Gp} \langle X\mid  x_{1}x_{2}=x_{2}x_{3}\,;\, x_{1}x_{3}=x_{4}x_{4}\,;\,
	x_{2}x_{1}=x_{4}x_{3}\,;\, x_{2}x_{2}=x_{3}x_{1}\,;\,
	x_{1}x_{4}=x_{3}x_{2}\,;\,x_{3}x_{4}=x_{4}x_{1} \rangle$ is the structure of $(X,r)$, a non-degenerate involutive irretractable set-theoretic solution.	The map $D$ satisfies $D(x_1)=x_1,\  D(x_2)=x_4,\, D(x_3)=x_3,\ D(x_4)=x_2$.
\end{ex}    

\subsection{The Yang-Baxter equation and Braces}

In \cite{rump_braces}, Rump introduced braces as a generalization of radical rings related with non-degenerate 
involutive set-theoretic solutions of the QYBE. In subsequent papers, he developed the theory of this new algebraic structure. In \cite{brace}, the authors give another equivalent definition of a brace and study its structure.  
\begin{defn}\cite{brace}
	A \emph{left brace}  is a set  $G$ with two operations, $+$ and $\cdot$, such that $(G+)$ is an abelian group, $(G,\cdot)$ is a group and for every $a,b,c \in G$:
\begin{equation}\label{eqn-brace}
	a \cdot (b+c) = a \cdot b+a \cdot c -a
\end{equation}
	The groups  $(G,+)$  and $(G,\cdot)$ are called \emph{the additive group} and \emph{the multiplicative group of the brace},  respectively.\\
	A right brace is defined similarly, by replacing Equation \ref{eqn-brace} by 
	\begin{equation}\label{eqn-brace-right}
(a+b) \cdot c +c = a \cdot c+b \cdot c
	\end{equation} 
\end{defn}
A \emph{two-sided brace} is a left and right brace, that is both Equations \ref{eqn-brace} and  \ref{eqn-brace-right} are satisfied. From the definition of a left brace $G$, it  follows  that the  identity of the multiplicative group of $G$ is equal to the identity  of the additive group of $G$. Additionnally,  for every $a,b,c \in G$, $a\cdot (b-c)=a\cdot b-a \cdot c +a$.
\begin{lem}\cite{brace}
	Let $G$ be a left brace.  Let $\lambda: G\rightarrow \operatorname{Aut}(G,+)$  be the map defined by $\lambda(a)=\lambda_a$  such that $\lambda_a(b)=a \cdot b -a$,  $a,b \in G$. Then  $\lambda$ satisfies the following properties:
	\begin{enumerate}[(i)]
		\item   $\lambda_{a \cdot b}= \lambda_a\lambda_b$, that is  $\lambda$ is a homomorphism of  (multiplicative) groups.
		\item $a \cdot \lambda_a^{-1}(b)=b\cdot \lambda_b^{-1}(a)$.			
		
		\item $\lambda_a\lambda_{\lambda_a^{-1}(b)}=\lambda_b\lambda_{\lambda_b^{-1}(a)}$.			
	\end{enumerate}
\end{lem}
\begin{defn}
	Let $G$  be a left brace. A non-empty subset $I \subseteq G$ is an  \emph{ideal of $G$}  if  $I$  is a normal subgroup of  $(G,\cdot)$   that satisfies:  for all $a\in G$, $b \in I$, $\lambda_a(b)\in I$.
\end{defn}
The following two theorems describe the correspondence between left braces and non-degenerate involutive set-theoretic solutions of the YBE.

\begin{thm}\cite{brace}, \cite{rump_braces}
	Let $G$ be a left brace. Then 
	\begin{enumerate}[(i)]
		\item  the map $r: G\times G \rightarrow G \times G$ defined by $r(x,y)=(\lambda_x(y), \,\lambda^{-1}_{\lambda_x(y)}(x))$ is a  non-degenerate involutive set-theoretic solution of the Yang-Baxter equation,  called the solution associated to the left brace $G$,  and denoted by $(G,r)$.
		\item  there exists a  non-degenerate involutive set-theoretic solution $(X,r')$ such that $Ret(X,r') \simeq (G,r)$ and moreover $\mathcal{G}(X,r')$ is isomorphic to the multiplicative group of the left brace $G$, where $\mathcal{G}(X,r')$ is the IYB group corresponding to tha solution $(X,r')$.
	\end{enumerate}
\end{thm}
There is also a converse:
\begin{thm}\cite{brace}
	Let $(X,s)$ be a  non-degenerate involutive set-theoretic solution. Then $G(X,s)$  is isomorphic to the multiplicative group of  a left brace $H$, and   $(X,s) \simeq(Y,r\mid_{Y^2})$,
	where $(H,r)$  is the solution associated to the left brace $H$and $Y\subseteq H$. 
\end{thm}
Roughly, a   group $G$ is the structure group of a solution if and only if it is isomorphic to the multiplicative 
group of a left brace $B$ such that the additive group of $B$ is a free abelian group with basis $X$ such that $\lambda_x(y) \in X$, for all $x,y \in X$ \cite{brace,rump_braces}.

\begin{defn}
	Let $ B_1$  and $B_2$ be two left braces. A map $\varphi: B_1 \rightarrow B_2$ 
	is a homomorphism of left braces if $\varphi(a+b)=\varphi(a)\ +  \varphi(b)$ and $\varphi(a\cdot b)=\varphi(a)\cdot \varphi(b)$, for every $a,b \in B_1$.
	The kernel of $\varphi$ is defined as $\operatorname{Ker}(\varphi)=\{a \in B_1 \mid\varphi(a)=1\}$ and it is an ideal of $B_1$.
\end{defn}

\begin{defn}
	Let $G$ be a left brace. The \emph{socle of $G$} is 
	\[\operatorname{Soc}(G)=\{a \in G \mid \lambda_a=Id\}=\{a \in G \mid a\cdot b=a+b; \; \forall b\in G\}\]
\end{defn}
The socle of $G$, $\operatorname{Soc}(G)$, is an ideal of $G$, since  $\operatorname{Soc}(G)$ is a normal subgroup of the multiplicative group of $G$ and it satisfies  $\lambda_{\lambda_a(b)}=Id$, 
for all $a\in G$, $b \in I$.

\begin{rem}
From Theorem \ref{prop-etingof},  there is a  monomorphism of groups $\tilde{\Phi} : G(X,r) \rightarrow \mathbb{Z}^{X} \rtimes \operatorname{Sym}_X$ and the group $G(X,r) $ is isomorphic to  a subgroup of  $\mathbb{Z}^{X} \rtimes \operatorname{Sym}_X$ of the form $H=\{(a,\varphi(a))\mid a \in \mathbb{Z}^{X} \}$, where $\varphi: \mathbb{Z}^{X} \rightarrow \operatorname{Sym}_X$, is defined by $\varphi(a)=\phi(\pi^{-1}(a))$.
The product in  $H$ is defined by $(a, \varphi(a))(b,\varphi(b))=(a+\varphi(a)(b), \varphi(a)\varphi(b))$. By defining a sum $\oplus$ on $H$  by the rule $(a, \varphi(a))\oplus(b,\varphi(b))=(a+b, \varphi(a+b))$,  the  map $\tilde{\Phi}$ can be extended to a homomorphism of left braces. Furthermore,  the IYB group   $\mathcal{G}(X,r)=\phi(G(X,r))$ has a unique structure of left brace such that $\phi$  is a homomorphism of left braces, with $\operatorname{Ker}(\phi)=\operatorname{Soc}(G(X,r))$. We refer to \cite[p.59]{cedo} for more details.
\end{rem}
%%%%%%%%%%%%%%%%%%%%%%%%%%%%%%%%%%%%%%%%%%%%%%%%%%%%%%%%%%%%%%

\subsection{Preliminaries on Garside monoids and groups} \label{sec_bcgd_gars}
\cite{deh_francais},  \cite{ddkgm}, \cite{DiM}
%%%%%%%%%%%%%%%%%%%%%%%%%%%%%%%%%%%%%%%%%%%%%%%
A monoid~$M$ is \emph{cancellative} if for every~${e,f,g,h}$ in $M$, the equality~$efg = ehg$ implies~$f = h$. The element $f$ is a \emph{right divisor}  ({\it resp.} a \emph{left divisor}) of $h$ if there is an element $z$ in $M$ such that $h = zf$  ({\it resp.} $h=fz$).  The element $h$ is a \emph{least common multiple (lcm) w.r to right-divisibility}  of $f$ and  $g$ in $M$ if $f$ and $g$ are right divisors of $h$ and
additionally  if there is an element $h'$ such that $f$ and $g$ are right divisors of $h'$, then
$h$ is  right divisor of $h'$. 
The element $a \in M$ is the \emph{greatest common divisor (gcd) w.r to right-divisibility} of $f$ and  $g$ in $M$ if $a$ right divides $f$ and $g$ and
additionally  if there is an element $a'$ that right divides $f$ and $g$, then
$a'$ is  right divisor of $a$.  The  lcm and gcd   with respect to left divisibility are defined similarly. A monoid~$M$  is right noetherian ({\it resp.}  \emph{left noetherian}) if every sequence~$(g_n)_{n\in\mathbb{N}}$ of elements of $M$ such that $g_{n+1}$ is a right divisor ({\it resp.} a left divisor) of $g_n$ stabilizes. It is noetherian if it is both left and right noetherian. If~$M$ is cancellative and noetherian, then left and right divisibilities are partial orders on~$M$.

\begin{defn} \begin{enumerate}[(i)]
		\item A \emph{locally Garside monoid} $M$ is a cancellative noetherian monoid such that
		any two elements in $M$ have a lcm for left (and right) divisibility if and only if they have a common multiple for left (and right) divisibility. 
		\item An element~$\Delta$ is \emph{balanced} if it has the same set of right and left divisors, denoted by $\operatorname{Div}(\Delta)$.
		\item  A \emph{Garside element} of  $M$ is a balanced element $\Delta$ with $\operatorname{Div}(\Delta)$ generating  $M$. 
		\item A monoid $M$ is a \emph{Garside monoid} if $M$ is  locally Garside with a Garside element $\Delta$ satisfying $\operatorname{Div}(\Delta)$ is finite. A group is a \emph{Garside group} if it is the group of fractions of  a Garside monoid.
	\end{enumerate}
\end{defn}
A  \emph{quasi-Garside monoid}  $M$ is defined as an extension of a Garside monoid  with  $\operatorname{Div}(\Delta)$  not necessarily finite \cite{quasi-garside}. The seminal examples of Garside groups are the braid groups and finite-type Artin groups  \cite{DePa}.  Torus link groups are also Garside groups \cite{picantin}.

\section{Description of the Garside structure  specific to the  structure group of a set-theoretic solution}  \label{secQYBE_Garside}
Structure  groups of non-degenerate, involutive set-theoretic solutions of the QYBE are Garside groups,  that satisfy  interesting properties \cite{chou_art,chou_godel1,chou_godel2,chou_left_garside}, \cite{ddkgm},  \cite{gateva}. We present some of the Garsideness properties needed here.
In the following, we assume  $(X,r)$ is  a  non-degenerate involutive set-theoretic solution with  structure group $G(X,r)$, $\mid X \mid =n$.
\begin{prop}\cite{chou_art}
 	\begin{enumerate}[(i)]
		\item  The lcm of $X$ for both left and right divisibilities is a Garside element, denoted by $\Delta$.
		\item The Garside element  $\Delta$ has length $\mid X \mid$ and $\mid \operatorname{Div}(\Delta)\mid = 2^n$. 
		\item $s \in \operatorname{Div}(\Delta)$ is the  lcm w.r to left-divisibility of~$X_{\ell} \subseteq X$ and   the  lcm w.r to right-divisibility of $X_r \subseteq X$,  with $\mid X_{\ell} \mid$ and $\mid X_r\mid$ equal to the length of $s$. 
	%	\item for any pair $(x_i,x_j)$, $x_i,x_j \in X$, the complement $x_i\setminus x_j$ is equal to $f_{i}^{-1}(j)$.
	\end{enumerate}	
\end{prop}

In Example \ref{exemple:exesolu_et_gars}, $\Delta=(x_1x_3)^2=(x_2x_4)^2=...=(x_3x_2)^2$ is the left and right lcm of $X$.  

\begin{defn}\label{defn_class}
	$(i)$ We say that $(X,r)$  \emph{satisfies $(C)$}, if $\sigma_x\sigma_y=Id_X$, whenever $r(x,y)=(x,y)$;  $xy$ and $yx$ are called \emph{frozen elements of length $2$}  \cite{chou_godel2}.\\
	$(ii)$ We say that $(X,r)$ is of \emph{class $m$}, if $m$  is the  minimal  natural number such that \\
	$\sigma_x\,\sigma_{D(x)}\,\sigma_{D^{2}(x)}\,...\,\sigma_{D^{m-1}(x)}\,=\,Id_X$,  for every $x \in X$.
\end{defn}

\begin{rem}\label{remark_defn_class}
	In  Definition \ref{defn_class}$(ii)$, we use the terminology from   \cite{deh_coxeterlike}, with a different (but equivalent) formulation. Being of class $2$ is equivalent to satisfying  $(C)$, with $(C)$ rewritten as  $\sigma_x\sigma_{\sigma_x^{-1}(x)}=Id_X$,  $\forall x \in X$.
\end{rem}

In  \cite{chou_godel2}, we show that if $(X, r)$ is  a  non-degenerate involutive set-theoretic solution, with $\mid X \mid =n$, that satisfies $(C)$,  then there is a short exact sequence $1 \rightarrow N \rightarrow G  \rightarrow W  \rightarrow 1$, where $N$ is a normal free abelian subgroup of rank $n$ and $W$ is a finite group of order $2^n$. Moreover, $W$  is  a Coxeter-like group, that is $W$ is  a finite quotient group with its elements in bijection with   $\operatorname{Div(\Delta^{})}$ \cite{chou_godel2}. In \cite{deh_coxeterlike}, P. Dehornoy  proves that  the condition $C$ may be dropped and  that for each  $(X,r)$,  there  exists  a minimal natural number $m$ such that for each $x \in X$  	$\sigma_x\,\sigma_{D(x)}\,\sigma_{D^{2}(x)}\,...\,\sigma_{D^{m-1}(x)}\,=\,Id_X$. This condition implies that for every  $x \in X$,  there is a chain of trivial relations of the form $xy_1=xy_1,\;y_1y_2=y_1y_2,\;y_2y_3=y_2y_3,...,y_{m-1}x=y_{m-1}x$, $y_i \in X$ and  there exists a (unique) element of length $m$ of the form $xy_1y_2y_3...y_{m-1}$, that   we  call the \emph{frozen element of length $m$ (starting with $x$)} and denote it by $\theta_x$. The subgroup $N$, generated by the $n$ frozen elements of length $m$,    is  normal, free abelian  of rank $n$ and the group $W$ defined by $G(X,S)/N$ is finite of order $m^n$ and is a  Coxeter-like  group:  the elements in $W$ are in bijection with $\operatorname{Div(\Delta^{m-1})}$ \cite{deh_coxeterlike}. The elements in $\operatorname{Div(\Delta^{m-1})}$ are called \emph{the simple elements}, and $\operatorname{Div(\Delta^{m-1})}$ is called a \emph{germ} \cite{deh-digne}.

\begin{ex} 
	In Example \ref{exemple:exesolu_et_gars},  the four trivial relations are  $x^{2}_{1}=x^{2}_{1},\,x^{2}_{3}=x^{2}_{3},\,x_{2}x_{4}=x_{2}x_{4},\,x_{4}x_{2}=x_{4}x_{2}$. The solution satisfies  $(C)$, as $\sigma_1^2=\sigma_3^2=\sigma_3\sigma_4=\sigma_4\sigma_3=Id_X$ and  $N$ is  generated by the four frozen elements of length two: $\theta_1=x_{1}^2,\; \theta_2=x_{2}x_4,\;\theta_3=x_{3}x_{3},\;\theta_4=x_{4}x_{2}$. The coxeter-like quotient group $W$ has order $2^4$.
\end{ex}
 Let $\mathbb{Z}^{n}$ the free abelian group generated  by $X$  and $\pi: G(X,r) \rightarrow \mathbb{Z}^{n}$ be the bijective $1$-cocycle defined in Theorem \ref{prop-etingof}.    
\begin{defn}
Let $g \in G(X,r)$.  We denote by $[g]_i$  the $i$-th coordinate of $\pi(g)$. 
\end{defn} 
Clearly,  for every $1 \leq i \leq n$, $[1]_i=0$, since we always assume $\pi(1)=0$.
In the following Lemmas, we give some computation rules for $\pi$  that are implicit in \cite{etingof} and appear in  \cite{deh_coxeterlike}  with a different notation (also in  \cite{chou_aut}). In the  first lemma, we 
describe the behaviour of $\pi$ on frozen elements.
\begin{lem}\label{lem_pi_frozen}
	Assume $(X,r)$ is of class $m$. Let $x \in X$, with $\pi(x)=t$.  Let $\theta_x$  be  the frozen element of length $m$ starting with $x$.
	\begin{enumerate}[(i)]
		\item  $r(x,D(x))=(x,D(x))$.
		\item  $\theta_x=$  $x\,D(x) \,...D^{m-2}(x)\,D^{m-1}(x)$. 
			\item  $\pi(\theta_x)\,=mt$.
			\item  $\pi(x\,D(x)\,...D^{k-1}(x))= kt$.
				\item  $\pi(\theta_{i_1}^{\alpha_1}...\theta_{i_k}^{\alpha_k})\,=m\,\sum\limits_{j=1}^{j=k}\alpha_{j}t_{i_j}$.
					\item  $\pi(\theta_{1}..\theta_{n})\,=\pi(\Delta^m)= m\,\sum\limits_{i=1}^{i=n}t_i$.
	\end{enumerate}   
\end{lem} 	
From Lemma \ref{lem_pi_frozen}$(v)$, for every $ 1\leq i \leq n$,   $[g]_i \equiv 0 \;(\operatorname{mod}  m)$ if and only $g \in N$.  In this lemma, we describe the behaviour of $\pi$ on simple elements.			
	\begin{lem}\label{lem_pi-simples}
		Assume $(X,r)$ is of class $m$.  Let $\mathcal{S}=\operatorname{Div}(\Delta^{m-1})$, $s \in \mathcal{S}$. Let $L_s$ denote the set of left divisors of $s$. Then $\pi(s)= \sum\limits_{i=1}^{i=n}\,[s]_i\,t_i$, where  $0 \leq [s]_i \leq m-1$. Furthermore:
		
		\begin{enumerate}[(i)]			
			\item    $[s]_i \neq 0$ if and only if $x_i \in L_s$.
			\item    $[s]_i =1$ if and only $x_i \in L_s$ and for every $k>1$, $x_i\,D(x_i)\,...D^{k-1}(x_i) \notin L_s$.
				\item    $[s]_i =2$ if and only $x_i\,D(x_i) \in L_s$ and for every $k > 2$, $x_i\,D(x_i)\,...D^{k-1}(x_i) \notin L_s$.
				More generally,  $[s]_i =\ell$  if and only $x_i\,D(x_i)\,...D^{\ell-1}(x_i) \in L_s$
		 and for every  integer  $k >\ell$, $x_i\,D(x_i)\,...D^{k-1}(x_i) \notin L_s$.
	\end{enumerate}

\end{lem} 
 Note that if  $s \in \operatorname{Div}(\Delta^{})$, then $\pi(s)= \sum\limits_{i\in I}t_i$, where $I=\{1 \leq i \leq n$$\mid\;$$x_i$  is a left divisor of $s\}$.
Lemmas  \ref{lem_pi_frozen} and \ref{lem_pi-simples} can be  rewritten in terms of left braces, as  it appears in \cite[p.70-77]{cedo}.  

\begin{rem}\label{rem-N-contained-ker}
	From lemma \ref{lem_pi_frozen}$(ii)$, and the definition of the class of the solution, $\phi(\theta_x)=Id_X$. That is, the normal subgroup $N \subset K$, where $K$ denotes  $\operatorname{Ker}(\phi)$  (or $\operatorname{Soc}(G(X,r)))$.  In $G(X,r)$,  there are $\mid\mathcal{G}\mid$ cosets of $K$  and $m^n$ cosets of $N$, so there are $\frac{m^n}{\mid\mathcal{G}\mid}$ cosets of $N$  in $K$, where $\mathcal{G}$ denotes the IYB group $\mathcal{G}(X,r)$. From the construction of the Coxeter-like quotient group, to each coset of $N$  in   $G(X,r)$,   there exists a unique representative of the coset that belongs to  $\operatorname{Div}(\Delta^{m-1})$ and so to each coset of $K$  there are   $\frac{m^n}{\mathcal{G}}$ representatives of the coset that belong to  $\operatorname{Div}(\Delta^{m-1})$.
\end{rem}

\begin{lem}\label{lem-simples-mod-m}
		Assume $(X,r)$ is of class $m$.  Let $\mathcal{S}=\operatorname{Div}(\Delta^{m-1})$.
	Let  $T_K=\{v_1=1,v_2,...,v_{\mid\mathcal{G}\mid }\}$ denote the set of  $\mid\mathcal{G}\mid$ representatives  of the cosets of   $K$ in $G(X,r)$ that belong to  $\mathcal{S}$. Let  $T=\{u_1=1,u_2,...,u_r\}$ denote the set of  $\frac{m^n}{\mid\mathcal{G}\mid}$ representatives  of $N$ in $K$  that belong to  $\mathcal{S}$. 
	Then, for every $u_p \in T$, $v_q \in T_K$, there exists  a unique $w_{p,q} \in \mathcal{S}$ such that $Nu_pv_q=Nw_{p,q}$.  Furthermore, $\pi(w_{p,q} )=\pi(u_p)+\pi(v_q)+ m\sum\limits_{i}^{}\alpha_it_i$, $\alpha_i\in \mathbb{Z}$.
\end{lem}
\begin{proof}
	From Remark \ref{rem-N-contained-ker}, to each coset of $N$  in   $G(X,r)$,   there exists a unique representative of the coset that belongs to  $\mathcal{S}$,  so to each coset $Nu_pv_q$,   there exists a unique representative  $w_{p,q} \in \mathcal{S}$ such that $Nw_{p,q}=Nu_pv_q$ and there is $g_{_{N}}\in N$, such that 
	$w_{p,q} =g_{_{N}}u_pv_q$. Moreover, $\pi(w_{p,q} )=\pi(g_{_{N}}u_p)+\pi(v_q)$, since $g_{_{N}}$ and $u_p$ act trivially and so, from Lemma \ref{lem_pi-simples},  $\pi(w_{p,q} )=\pi(u_p)+\pi(v_q)+ m\sum\limits_{i}^{}\alpha_it_i$, $\alpha_i\in \mathbb{Z}$.
\end{proof}

\section{Proof of the main result}

The structure group $G(X,r)$ of  a   non-degenerate, involutive set-theoretic solution $(X,r)$,  with $\mid X \mid =n$ is a Bieberbach group of rank  $n$, that  is a torsion-free crystallographic group (a discrete cocompact subgroup of  the group of  isometries of $\mathbb{R}^n$).  Indeed,  $G(X,r)$ acts freely on $\mathbb{R}^n$ by isometries  (see  \cite{gateva_van}, \cite{jespers_book}[p.218]). From  the action of $G(X,r)$  it results that $G(X,r)$  embeds in $GL_{n+1}(\mathbb{R})$ as a subgroup of matrices with a  permutation matrix  in the upper $n\times n$ block and a translation part in the last column in the following way:

\begin{figure}[h]
	\centering
\begin{equation}\label{eqn-psi-embed}
	\psi: \; G(X,r)  \; \rightarrow    \; GL_{n+1}(\mathbb{R}) 
\end{equation}
         \[  g      \mapsto 	 \begin{pmatrix}
	\begin{matrix}
A_g\\
		\end{matrix}
	&  \rvline &\pi(g)\\	
	\hline
		\begin{matrix}
	0&... &0\\
	\end{matrix}
		& \rvline
	& 1 \\
	\end{pmatrix} 
\]
\end{figure}
We denote by $A_g$ the $n \times n $ permutation matrix corresponding to $\phi(g)$, written as a representing matrix. As an example, the first column in the  permutation matrix of $(1,2,3)$ is  $(0,1,0)^t$. 
Note that  $G(X,r)$ embeds also  in the group $GL_{n+1}(\mathbb{Z})$ and that this embedding is the same as  the embedding  $\tilde{\Phi}: G(X,r) \rightarrow \mathbb{Z}^{X} \rtimes \operatorname{Sym}_X$ defined by $\tilde{\Phi} (g)=(\pi(g), \phi(g))$, written in a matrix form. However, it is not fit for our purpose as $\mathbb{Z}$ is not a division ring.

The monomorphism of groups  $\psi$ induces a representation of $G(X,r)$ in the algebra $M_{n+1}(\mathbb{R})$ and by linear extension $\psi$ induces  a  unital homomorphism of  $\mathbb{R}$-algebras $\hat{\psi}:  \mathbb{R}G  \rightarrow    M_{n+1}(\mathbb{R}) $, where $\mathbb{R}G$ denotes the group algebra of $G(X,r)$, defined by:

\begin{figure}[h]
	\centering
	\begin{equation}\label{eqn-psi-gp-alg}
\hat{\psi}: \; \mathbb{R}G  \; \rightarrow    \; M_{n+1}(\mathbb{R}) 
	\end{equation}
	\[  \sum\limits_{g \in G}k_g\,g  \;    \mapsto\; 	 \sum\limits_{g \in G}k_g\,\psi(g) 	\]
\end{figure}
The basis of   $\mathbb{R}G$ is the set of all group elements $g \in G(X,r)$ and every element in  $\mathbb{R}G$  can be written as    $\sum\limits_{g \in G}k_g\,g$, with $k_g\neq 0$ for only  finitely many elements $g \in G(X,r)$.
The subalgebra  of  $M_{n+1}(\mathbb{R}) $, $\operatorname{Im}(\hat{\psi})=\{\sum\limits_{g \in G}k_g\,\psi(g) 	 \mid k_g \in \mathbb{R}\}$  is spanned by  the infinite set $\{\psi(g) 	\mid g \in G(X,r)\}$. However,  we show that it is finite-dimensional and that it has a finite basis composed of images of some simples and of frozen elements.\\
We define  the set $\{\mathcal{E}_k \mid 1 \leq k \leq n\}$ $\subset$  $M_{n+1}(\mathbb{R}) $:
\begin{gather*}
(\mathcal{E}_k)_{k,n+1}=1 \;\;   \\
(\mathcal{E}_k)_{i,j}=0, \;\; i \neq k \; ;\;  j \neq n+1
\end{gather*}
\begin{lem}\label{lem-all-images-product-simple}
		Let $m \geq 2$ be the class of the solution $(X,r)$. Let $N=\langle \theta_1,...,\theta_n\rangle$ be the normal subgroup of  $G(X,r)$ generated by the $n$ frozen elements of length $m$.  Let 	$\mathcal{S}=\operatorname{Div}(\Delta^{m-1})$ be the set of simples.  Let $g\in G(X,r)$ . Then 
		\begin{enumerate}[(i)]
				\item  	There exist $s \in \mathcal{S} $  and $g_{_{ N} }\in N$, $g_{_{ N} } =\prod\limits_{j=1}^{j=k}\theta_{i_j} ^{\alpha_{i_j}}$,  $\alpha_{i_j}\in \mathbb{Z}$,  such that $ \psi(g)=\psi(g_{_{ N} })\,\psi(s)$.
			\item  $ \psi(g)=\psi(s)\,+\; m\sum\limits_{j=1}^{j=k}\alpha_{i_j}\,\mathcal{E}_{i_j}$.
		\end{enumerate}

\end{lem}
\begin{proof}
$(i)$	From Section \ref{secQYBE_Garside}, $N$ is a subgroup of finite index $m^n$ and $G(X,r)$ is the disjoint union of  $m^n$ cosets of $N$ of the form $Ns$, with $s \in \mathcal{S} $. So,  there exist $g_{_{ N} } \in N$ and $s \in \mathcal{S} $ such that $g=g_{_{ N} }s$, and since $\psi$ is a multiplicative homomorphism,  $\psi(g)=\psi(g_{_{ N} })\,\psi(s)$.  \\
$(ii)$ From the definition of $N$, $g_{_{ N} } =\prod\limits_{j=1}^{j=k}\theta_{i_j} ^{\alpha_{i_j}}$,  $\alpha_{i_j}\in \mathbb{Z}$,  and from Remark \ref{rem-N-contained-ker},  $\phi(g_{_{ N} })=Id_X$ and $\pi(g_{_{ N} })=m\sum\limits_{j=1}^{j=k}\alpha_{i_j}\,t_{i_j}$.
	From $(i)$, $ \psi(g)=\psi(g_{_{ N} })\,\psi(s)$ and this  product  can be written as a  sum of matrices:  $  	 \begin{pmatrix}
	\begin{matrix}
	A_s\\
	\end{matrix}
	&  \rvline &\pi(s)+\pi(g_{_{ N} })\\	
	\hline
	\begin{matrix}
	0&... &0\\
	\end{matrix}
	& \rvline
	& 1 \\
	\end{pmatrix} =
	\begin{pmatrix}
	\begin{matrix}
	A_s\\
	\end{matrix}
	&  \rvline &\pi(s)\\	
	\hline
	\begin{matrix}
	0&... &0\\
	\end{matrix}
	& \rvline
	& 1 \\
	\end{pmatrix}+
	\begin{pmatrix}
	\begin{matrix}
	0^{n\times n}\\
	\end{matrix}
	&  \rvline & \pi(g_{_{ N} })\\	
	\hline
	\begin{matrix}
	0&... &0\\
	\end{matrix}
	& \rvline
	& 0 \\
	\end{pmatrix}$.\\
	That is,  $ \psi(g)=\psi(s)\,+\; m\sum\limits_{j=1}^{j=k}\alpha_{i_j}\,\mathcal{E}_{i_j}$.
	\end{proof}

\begin{lem}\label{lem-ei-belong-simples}
Let $m \geq 2$ be the class of the solution $(X,r)$.   Then,  the set $\{\mathcal{E}_k  \mid 1 \leq k \leq n\}$ belongs to $\operatorname{Span} (\{\psi(\theta_k)\mid  1 \leq k\leq n\}\cup\{\psi(1)\})$.
	\end{lem}
\begin{proof}
From the definition of $\psi$, $  	\psi(1)= \begin{pmatrix}
\begin{matrix}
I^{n\times n}\\
\end{matrix}
&  \rvline & 0^{n \times 1}\\	
\hline
\begin{matrix}
0&... &0\\
\end{matrix}
& \rvline
& 1 \\
\end{pmatrix} $ and $	\psi(\theta_k)= \begin{pmatrix}
	\begin{matrix}
		I^{n\times n}\\
	\end{matrix}
	&  \rvline & \pi(\theta_k)\\	
	\hline
	\begin{matrix}
		0&... &0\\
	\end{matrix}
	& \rvline
	& 1 \\
\end{pmatrix} $,  $ 1 \leq k \leq n$.  So,  from Lemma \ref{lem_pi-simples},  $\psi(\theta_k)=\psi(1) + m\mathcal{E}_k$.  That is,   $\mathcal{E}_k = \frac{1}{m}(\psi(\theta_k)-\psi(1) )$, which implies  $\{\mathcal{E}_k  \mid 1 \leq k \leq n\}$ belongs to $\operatorname{Span} (\{\psi(\theta_k)\mid  1 \leq k\leq n\}\cup\{\psi(1)\})$.
\end{proof}
We reformulate now Theorem \ref{theo} in a more precise way and give its proof.
\begin{thm}\label{thm}
		Let $(X,r)$ be  a  non-degenerate involutive set-theoretic solution, with  $\mid X \mid =n$, and structure group $G(X,r)$.  Let $m \geq 2$ be the class of the solution $(X,r)$.    Let $\mathcal{S}=\operatorname{Div}(\Delta^{m-1})$.
		Let  $T_K=\{v_1=1,v_2,...,v_{\mid\mathcal{G}\mid }\}$ denote the set of  $\mid\mathcal{G}\mid$ representatives  of the cosets of   $K$ in $G(X,r)$ that belong to  $\mathcal{S}$.  Let $ \theta_1,...,\theta_n$  the $n$ frozen elements of length $m$.  Then,  
		$\operatorname{Im}(\hat{\psi})$ is a  finite-dimensional real vector space and it has a basis a subset of  
			the set  $\{\psi(\theta_k)\mid  1 \leq k\leq n\} \cup \{\psi(v_q)\mid v_q \in T_K\})$. Moreover, the dimension of $\operatorname{Im}(\hat{\psi})$  is at most  $n +\mid \mathcal{G}\mid$.
\end{thm}
\begin{proof}
	From Lemma \ref{lem-all-images-product-simple}, for every $g \in G(X,r)$, there exists $ s \in \mathcal{S}$ such that $ \psi(g)=\psi(s)\,+\; m\sum\limits_{j=1}^{j=k}\alpha_{i_j}\,\mathcal{E}_{i_j}$. From Lemma \ref{lem-ei-belong-simples},  $\mathcal{E}_{i_j} $  belong to $\operatorname{Span} (\{\psi(\theta_k)\mid  1 \leq k\leq n\}\cup\{\psi(1)\})$.  So,  $ \psi(g)$ belongs to the span of  $\{\psi(\theta_k)\mid  1 \leq k\leq n\}\cup\{\psi(s) \mid s \in \mathcal{S}\})$. It remains to show $\{\psi(s) \mid s \in \mathcal{S}\}$ is contained in  $\operatorname{Span} (\{\psi(\theta_k)\mid  1 \leq k\leq n\} \cup \{\psi(v_q)\mid v_q \in T_K\})$. Every $s \in \mathcal{S}$ is a unique representative of a coset of $N$ in $G(X,r)$, so  there exist  $u_p \in T$, $v_q \in T_K$  such that $Ns=Nu_pv_q$, and from Lemma \ref{lem-simples-mod-m},  $\pi(s )=\pi(u_p)+\pi(v_q)+ m\sum\limits_{i}^{}\alpha_it_i$. So, $\psi(s)=  \begin{pmatrix}
\begin{matrix}
A_{v_q}\\
\end{matrix}
&  \rvline &\pi(s )\\	
\hline
\begin{matrix}
0&... &0\\
\end{matrix}
& \rvline
& 1 \\
\end{pmatrix} =
\begin{pmatrix}
\begin{matrix}
A_{v_q}\\
\end{matrix}
&  \rvline &\pi(v_q)\\	
\hline
\begin{matrix}
0&... &0\\
\end{matrix}
& \rvline
& 1 \\
\end{pmatrix}+
\begin{pmatrix}
\begin{matrix}
0^{n\times n}\\
\end{matrix}
&  \rvline & \pi(u_p)\\	
\hline
\begin{matrix}
0&... &0\\
\end{matrix}
& \rvline
& 0 \\
\end{pmatrix} +m \sum\limits_{i}^{}\alpha_i\mathcal{E}_i=
\psi(v_q)+\sum\limits_{j}^{}\beta_j\mathcal{E}_j+ m \sum\limits_{i}^{}\alpha_i\mathcal{E}_i$. So, 	the set  $\{\psi(\theta_k)\mid  1 \leq k\leq n\} \cup \{\psi(v_q)\mid v_q \in T_K\}$ spans $\operatorname{Im}(\hat{\psi})$. As there are $n$ frozen elements of length $m$ and $\mid \mathcal{G}\mid$  elements in $T_K$,  the dimension of $\operatorname{Im}(\hat{\psi})$  is less equal than $n +\mid \mathcal{G}\mid$.
\end{proof}

 In the case of Example  \ref{exemple:exesolu_et_gars},  the  group  $\mathcal{G}(X,r)$ has order 8 and 	$\operatorname{Im}(\hat{\psi})$  is spanned by: \\
  $\{\psi(1), \psi(x_1), \psi(x_2),\psi(x_3),\psi(x_4),  \psi(x_1x_2),\psi(x_1x_3),\psi(x_1x_4)\}\cup \{\psi(x_1^2),\psi(x_2x_4),\psi(x_3^2), \psi(x_4x_2)\}$. This spanning set of  $12$ elements is linearly dependant and the dimension of 	$\operatorname{Im}(\hat{\psi})$  is $10$.  In the case of the (permutation) solution $(X,r)$ with $X=\{x_1,x_2,x_3\}$, $\sigma_1=\sigma_2=\sigma_3=(1,3,2)$, and  with structure group $G(X,r)=\operatorname{Gp}\langle x_1^2=x_3x_2, x_2^2=x_1x_3, x_3^2=x_2x_1\rangle$. The solution is  of class $3$, with the  frozen elements of length $3$:   $\theta_1=x_1x_2x_3,\theta_2=x_2x_3x_1, \theta_3=x_3x_1x_2$. 
The  group  $\mathcal{G}(X,r)$ has order 3  and  	$\operatorname{Im}(\hat{\psi})$  is  spanned by  the  six  linearly independant  elements: 
$\{\psi(1), \psi(x_1), \psi(x_2x_1)\}\cup\{\psi(x_1x_2x_3),\psi(x_2x_3x_1), \psi(x_3x_1x_2)\}$. The dimension of $\operatorname{Im}(\hat{\psi})$  is $6$.\\

The embedding of groups $\tilde{\Phi} : G(X,r) \rightarrow \mathbb{Z}^{X} \rtimes \operatorname{Sym}_X$
  is equivalent to the  embedding 
$\psi_{_{\mathbb{Z}}}: \; G(X,r)  \; \rightarrow    \; GL_{n+1}(\mathbb{Z}) $,  with  $\psi_{_{\mathbb{Z}}}$ defined in the same way as in  Equation  \ref{eqn-psi-embed}.  From Remark \ref{rem-N-contained-ker}, 
 $\tilde{\Phi}$   extends to a homomorphism of left braces, with the definition of  $\oplus$ on $\operatorname{Im}(\tilde{\Phi} ( G(X,r)))$ by   $(a, \varphi(a))\oplus(b,\varphi(b))=(a+b, \varphi(a+b))$, with $\varphi(a+b)=\phi(\pi^{-1}(a+b))$. So, the same can be said for $\psi_{_{\mathbb{Z}}}$ and a question that arises naturally is whether Theorem \ref{thm} can be obtained directly from these considerations. It seems the answer is  negative, since the definition of 
$\oplus$ on $\operatorname{Im}(\tilde{\Phi} ( G(X,r)))$   induces  the definition of  a  sum $\oplus$  on $\operatorname{Im}(\psi_{_{\mathbb{Z}}}( G(X,r)))$, different from the standard sum of matrices, and the  notion of spanning set is not the usual one anymore.
To illustrate how this new sum looks for matrices, we compute it  for two simple examples:
 \begin{gather*}
  \psi_{_{\mathbb{Z}}}( x_i)\oplus \psi_{_{\mathbb{Z}}}( x_j)=
 	 \begin{pmatrix}
\begin{matrix}
A_{\sigma_i}\\
\end{matrix}
&  \rvline & t_i\\	
\hline
\begin{matrix}
0&... &0\\
\end{matrix}
& \rvline
& 1 \\
\end{pmatrix} \oplus
\begin{pmatrix}
\begin{matrix}
A_{\sigma_j}\\
\end{matrix}
&  \rvline & t_j\\	
\hline
\begin{matrix}
0&... &0\\
\end{matrix}
& \rvline
& 1 \\
\end{pmatrix}=
\begin{pmatrix}
\begin{matrix}
A_{\sigma_i\sigma_i^{-1}(j)}\\
\end{matrix}
&  \rvline &t_i+t_j\\	
\hline
\begin{matrix}
0&... &0\\
\end{matrix}
& \rvline
&1 \\
\end{pmatrix}\\
\psi_{_{\mathbb{Z}}}( x_i)\oplus \psi_{_{\mathbb{Z}}}( x_j)\oplus \psi_{_{\mathbb{Z}}}( x_k)=\begin{pmatrix}
\begin{matrix}
A_{ \sigma_i\sigma_{\sigma_i^{-1}(j)}\sigma_{\sigma^{-1}_{\sigma_i^{-1}(j)}\sigma_i^{-1}(k)}}\\
\end{matrix}
&  \rvline &t_i+t_j+t_k\\	
\hline
\begin{matrix}
0&... &0\\
\end{matrix}
& \rvline
&1 \\
\end{pmatrix}, \; x_i,x_j,x_k \in X.
\end{gather*}

To conclude, it would be interesting to know whether  	$\operatorname{Im}(\hat{\psi})$  can be spanned by images of simple elements only, as in the case for $B_k$.

%--------------------------

\bigskip\bigskip\noindent
{ Fabienne Chouraqui}

\smallskip\noindent
University of Haifa at Oranim, Israel.

\smallskip\noindent
E-mail: {\tt fabienne.chouraqui@gmail.com} \\

                {\tt fchoura@sci.haifa.ac.il}
\end{document}